\documentclass{amsart}
\usepackage{graphicx} 
\usepackage{amsmath,amssymb,amsfonts}
\usepackage{cleveref}
\usepackage{comment}

\newcommand{\CC}{\mathbb C}

\newcommand{\Tt}{\mathcal T}
\newcommand{\Nn}{\mathcal N}
\newcommand{\RR}{\mathbb R}
\newcommand{\HH}{\mathbb H}
\newcommand{\Ee}{\mathcal E}
\newcommand{\Tgn}[1][g,n]{\Tt_{#1}}

\newcommand{\Sgn}[1][g,n]{\Sigma_{#1}}
\newcommand\Mgn[1][g,n]{\mathcal{M}_{#1}}

\DeclareMathOperator{\PSL}{PSL}

\DeclareMathOperator{\WP}{WP}
\DeclareMathOperator{\Teich}{Teich}
\DeclareMathOperator{\Fix}{Fix}

\DeclareMathOperator{\Diff}{Diff}
\DeclareMathOperator{\orb}{orb}

\newtheorem{theorem}{Theorem}[section]

\DeclareMathOperator{\Mod}{Mod}
\DeclareMathOperator{\PMod}{PMod}

\newtheorem{corollary}[theorem]{Corollary}

\newtheorem{lemma}[theorem]{Lemma}

\newtheorem{proposition}[theorem]{Proposition}

\theoremstyle{definition}
\newtheorem{definition}[theorem]{Definition}
\newtheorem{remark}[theorem]{Remark}
\newtheorem{claim}[theorem]{Claim}

\newtheorem{question}[theorem]{Question}

\crefname{question}{question}{questions}

\AddToHook{env/claim/begin}{\crefalias{theorem}{claim}}
\AddToHook{env/remark/begin}{\crefalias{theorem}{remark}}
\AddToHook{env/lemma/begin}{\crefalias{theorem}{lemma}}
\AddToHook{env/proposition/begin}{\crefalias{theorem}{proposition}}
\AddToHook{env/definition/begin}{\crefalias{theorem}{definition}}
\AddToHook{env/corollary/begin}{\crefalias{theorem}{corollary}}
\AddToHook{env/question/begin}{\crefalias{theorem}{question}}

\usepackage{tikz}
\usepackage{tikz-cd}

\usepackage[alphabetic]{amsrefs}
\usepackage{a4wide}

\usepackage{setspace}

\title[Local rigidity of covering constructions and WP subvarieties of $\Mgn$]{Local Rigidity of covering constructions and Weil--Petersson subvarieties of the Moduli Space of Curves}
\author{Carlos A. Serv\'an}
\address{Department of Mathematics, University of Michigan}
\email{cservan@umich.edu}

\begin{document}
\begin{abstract} We show that totally geodesic subvarieties of the moduli space $\Mgn$ of genus $g$ curves with
  $n$ marked points,
  endowed with the Weil--Petersson metric, are locally rigid. This implies that
  covering constructions---examples of totally geodesic subvarieties of $\Mgn$ endowed with the Teichm\"uller metric---are locally rigid. We deduce the local rigidity statement
  from a more general rigidity result for a class of
  orbifold maps to $\Mgn$.
  \end{abstract}
\maketitle
  \section{Introduction}
  Suppose $3g - 3 + n>0$ and let $\pi:\Tgn \to \Mgn$
  be the projection from the Teichm\"uller space of a genus $g$ surface with $n$
  marked points to the associated moduli space.
  The Weil--Petersson metric $g_{\WP}$ on $\Tgn$ is a
  negatively curved K\"ahler metric that descends to a metric on $\Mgn$. The metric $g_{\WP}$
  induces on $\Mgn$ the structure of a quasi-projective variety~\cite{Wolpert-qp}.
  A class of subvarieties\footnote{Unless otherwise stated we always consider subvarieties in the algebraic category.} of $\Mgn$ naturally associated with $g_{\WP}$
  is the following:

  A complex submanifold $M \subset \Tgn$ is called \emph{Weil--Petersson} if it is totally
  geodesic with respect to the Weil--Petersson metric $g_{\WP}$. This means
  that the $g_{\WP}$-geodesic passing through any two points in $M$ is completely contained in $M$.
  Similarly, a subvariety $N \subset \Mgn$ is called \emph{Weil--Petersson} if an
  irreducible component $M$ of $\pi^{-1}(N)$ is a Weil--Petersson complex submanifold of $\Tgn$.
  The complex submanifold $M$ is called a \emph{lift} of $N$.
  More generally, we call a subvariety $N \subset \Mgn$ \emph{almost Weil--Petersson} if it has a lift $M \subset \Tgn$ that maps biholomorphically via a forgetful map $\Tgn \to \Tgn[g,m]$
  onto a Weil--Petersson complex submanifold of some $\Tgn[g,m]$ with $n\geq m$ and $3g-3+m>0$.

  We say that a map of complex
  orbifolds $f:X \to \Mgn$ is an \emph{almost Weil--Petersson immersion} if it is a
  proper immersion and $f(X)$ is an
  almost Weil--Petersson subvariety. In particular,
  this implies that $X$ is a quasi-projective variety. The normalization map
  ~\cite{GR}*{Chapter 8} of an almost Weil--Petersson subvariety provides the basic example of an
  of almost Weil--Petersson immersion (see \Cref{def:WP-im}). In this work we study
  the following rigidity property of this class of maps.

  A proper immersion $f:X \to \Mgn$ from a quasi-projective variety $X$ is \emph{locally rigid} if any holomorphic deformation, i.e. a holomorphically varying family of maps
  \[ f_t: X_t \to \Mgn \ \ , \ \ t \in \Delta \]
  with $(X_0,f_0) \cong (X, f: X \to \Mgn)$,
  through proper immersions $f_t$ and quasi-projective varieties $X_t$ is trivial: there is a holomorphic family of biholomorphisms $g_t:X_t \to X$ inducing
  $(X_t,f_t) \cong (X, f)$.
  Our maps should be understood in the orbifold sense, i.e. on a level $L\geq 3$ cover of $\Mgn$
  (see \Cref{def:proper_def} and \Cref{rmk:orbi-issues} for details). Our main theorem is the following.
  \begin{theorem}[\textbf{Local Rigidity}]\label{theo:main_rigidity}
    Suppose $3g-3 + n > 0$ and $\dim(X) \geq 1$. Then, any almost Weil--Petersson immersion
    $f:X \to \Mgn$ is locally rigid.
  \end{theorem}
  \begin{remark}
    The assumption on $X_t$ being quasi-projective is only used to get enough maps
    from finite type curves $C \to X_t$ (cf. \Cref{lemma:crap}).
  \end{remark}

\noindent \textbf{Examples: Covering constructions.} A source of examples of almost Weil--Petersson subvarieties
is given by \emph{covering constructions}~\cite{BS,MMW}, defined
as follows. Let
\[ h:\Sigma_{g',n'} \to \Sgn \] be
a finite degree, orientation-preserving, topological branched cover. Assume further that
the preimage of the marked points in $\Sgn$ equals the union of the marked points in $\Sgn[g',n']$ and
the ramification points of $h$. Pulling back complex structures under $h$ gives a holomorphic map
\[ f_h:\Tgn \to \Tgn[g',\ell] \]
for some $\ell \geq n'$. We call $f_h$ a \emph{totally marked} covering construction (see \Cref{section:covering_geodesic} for more details). A \emph{covering construction}
is a holomorphic map
\[ f: \Tgn \to \Tgn[g',n']\] given by the composition of a totally
marked covering construction $f_h:\Tgn \to \Tgn[g',\ell]$ with a forgetful map
$\Tgn[g',\ell] \to \Tgn[g',n']$ which forgets \emph{only} ramification points of $h$.
If the map $h$ is a regular branched cover, we call $f$ a \emph{regular covering construction}. The link between covering constructions
and almost Weil--Petersson submanifolds of $\Tgn$ is the following.
\begin{proposition}\label{lemma:almost_geodesic}
 Assume $3g-3 + n >0$. Let $f:\Tgn \to \Tgn[g',n']$ be a covering construction. The following hold:
  \begin{enumerate}
  \item Assume $f$ is totally marked. Then, the image $f(\Tgn)$  is a Weil--Petersson complex submanifold of $\Tgn[g',n']$.
  \item Assume $f$ is regular and not totally marked.
    If $(g,g') = (0,1)$ assume further that the marked points satisfy $n \geq 5$. Then, the image $f(\Tgn)$ is an almost Weil--Petersson complex
    submanifold of $\Tgn[g',n']$.
  \end{enumerate}
  In particular, in either case the projection $\pi(f(\Tgn)) \subset \Mgn[g',n']$ is an almost Weil--Petersson subvariety.
\end{proposition}
\begin{remark}
  To our knowledge, covering constructions provide the only known examples of almost Weil--Petersson subvarieties of $\Mgn$. Thus, we ask the following.
  \begin{question}\label{question:wp}
    Does there exist a Weil--Petersson subvariety $N \subset \Mgn$ which is not given by a covering construction?
  \end{question}

\end{remark}
We now give two applications of \Cref{theo:main_rigidity}.

\medskip \noindent \textbf{Application 1. Totally geodesic subvarieties for the Teichm\"uller metric.}
Replacing the Weil--Petersson metric by the \emph{Teichm\"uller metric}
$d_{\Teich}$ in our definition of Weil--Petersson subvarieties gives \emph{Teichm\"uller} subvarieties $N \subset \Mgn$ and
\emph{Teichm\"uller} complex submanifolds $M \subset \Tgn$~\cite{Wright,Wright-AH,BDR}.

  $1$-dimensional Teichm\"uller subvarieties are called \emph{Teichm\"uller curves}.
  McMullen showed~\cite{Mcmullen-Rigid} that the normalization map of a Teichm\"uller curves is rigid.\footnote{McMullen considered a more general class of
    maps that are the equivalent (and the inspiration) of our definition of an almost Weil--Petersson immersion for the
    Teichm\"uller metric. Furthermore, he considers a more general class of local deformations not assuming that
    $f_t$ is a proper immersion.
    In the one-dimensional case, our proof of \Cref{theo:main_rigidity} works without the extra assumptions on $f_t$ and so it gives
    an alternative proof of McMullen's result for the case of covering constructions. See \Cref{theo:rigidity-curves} for a stronger rigidity result in the one dimensional case.}
  Arana-Herrera--Wright~\cite{Wright-AH}*{Question 10.6} ask if the same holds for Teichm\"uller varieties of
  higher dimensions. \Cref{theo:main_rigidity} implies a positive answer to~\cite{Wright-AH}*{Question 10.6} in the case of covering constructions.

  \begin{corollary}[\textbf{Local rigidity of Covering constructions}]\label{cor:covering_rigidity} Assume $3g-3 + n >0$. Let $f:\Tgn \to \Tgn[g',n']$ be a totally marked or regular covering construction, and fix the notation and assumptions of \Cref{lemma:almost_geodesic}.
    The normalization map of $\pi(f(\Tgn)) \subset \Mgn[g',n']$ is locally rigid.
  \end{corollary}
  \begin{proof}
    The normalization map of an almost Weil--Petersson subvariety defines an almost Weil--Petersson immersion (see \Cref{section:deformation}), and so the claim follows
    from \Cref{theo:main_rigidity}.
  \end{proof}

\begin{remark}
  Any covering construction arises by applying forgetful maps to a totally marked covering construction.
  Thus, up to forgetful maps, all covering constructions are locally rigid.
\end{remark}

\begin{remark} Covering constructions encompass all of the known
  examples of Teichm\"uller subvarieties of $\Mgn$ of dimension bigger than $1$,
  except for the $2$-dimensional examples $\{Y_i\}$ found recently by McMullen, Mukamel, Wright and
Eskin~\cite{MMW,EMMW}. In this regard, the following
  questions are quite pertinent.
  \begin{question}\label{question:question1}
    Are the Teichm\"uller surfaces $Y_i$ totally geodesic with respect to the Weil--Petersson metric, i.e. are $Y_i$ also Weil--Petersson surfaces?
  \end{question}
  \begin{question}\label{question:question2}
    Given a subvariety $N \subset \Mgn$. Suppose that $N$ is totally geodesic with respect to
    \emph{both} the Teichm\"uller and the Weil--Petersson metric. Is $N$ given by a covering construction?
  \end{question}

\end{remark}

\noindent \textbf{Application 2. Characterization of almost Weil--Petersson submanifolds.}
The pure mapping class group
$\PMod(\Sgn):=\pi_0(\Diff^+(\Sgn,\{x_1,\ldots,x_n\}))$ acts on $\Tgn$ by biholomorphisms. For an arbitrary
analytic subset $W \subset \Tgn$, let $\Gamma_W$ be the stabilizer of $W$ in
$\PMod(\Sgn)$.
As a corollary of the proof of \Cref{theo:main_rigidity} (see \Cref{theo:theorem1}), we obtain
the following.
\begin{corollary}[\textbf{Maximality}]\label{corollary:characterization}
  Let $3g-3 + n > 0$ and be given an almost Weil--Petersson complex submanifold $M \subset \Tgn$ of positive dimension. Assume that $\pi(M) \subset \Mgn$ is a subvariety. Then, the submanifold
  $M$ is maximal among
  the analytic subsets $W$ of $\Tgn$
  such that the following hold:
  \begin{enumerate}
  \item $\Gamma_W \subseteq \Gamma_M.$
  \item $W/\Gamma_W$ is a quasi-projective variety.
  \end{enumerate}
  In particular, the conjugacy class of $\Gamma_M$ in $\PMod(\Sgn)$ is a complete invariant of $\pi(M)$.
\end{corollary}

\begin{remark}
  \Cref{corollary:characterization} showcases the following rigidity
  result for the pair $(M,\Gamma_M)$. Inclusion at the level of the stabilizers $\Gamma_W < \Gamma_M$ implies
  the inclusion of \emph{spaces} $W \subset M$. A similar
  statement holds for families of curves over quasi-projective varieties: if the monodromy factors through $\Gamma_M$, then the classifying map
  to $\Tgn$ factors through $M$ (see \Cref{theo:theorem1} for a precise statement).
\end{remark}

\noindent \textbf{Paper overview.}
Let $M \subset \Tgn $ be an almost Weil--Petersson lift of a subvariety of $\Mgn$, and let $\Gamma_M$
be the stabilizer of $M$ in $\PMod(\Sgn)$. In \Cref{section:prelim} we quickly review all the necessary material from
Teichm\"uller theory needed for the rest of the paper, focusing on the Weil--Petersson metric. In \Cref{section:deformation} we introduce the
concept of a $\Gamma_M$-deformation: a class of orbifold maps
$\phi:X \to \Mgn$ which includes the classifying maps
of families of curves over quasi-projective varieties whose
monodromy factors through $\Gamma_M$. \Cref{theo:main_rigidity}
is a consequence of the stronger rigidity result for $\Gamma_M$-deformations
given in \Cref{theo:theorem1}: \emph{any nonconstant $\Gamma_M$-deformation
factors through $M$}. \Cref{section:deformation} also shows
that $\Gamma_M$ is the orbifold fundamental group of the \emph{normalization} of $\pi(M)$, and
defines the notion
an almost Weil--Petersson immersion, which could be roughly stated as a finite orbifold covering of $M/\Gamma_M$. In \Cref{section:main_theo} we provide the details of the
proof of \Cref{theo:theorem1} in the case that $X$ is a finite type curve, starting with an elementary Riemannian geometry lemma, relying on fundamental results of Wolpert~\cite{Wolpert},
and then following the argument appearing in the proof of the Imayoshi--Shiga theorem
given in \cite{AAS}*{Section 4}. The key observation is that, due to the arguments of \cite{AAS},
\emph{holomorphic maps minimize energy in their homotopy class}.\footnote{For compact $C$
  the proof follows from the theory of harmonic maps (cf. ~\cite{ES}), but an extra step is needed when C is finite volume.}
 In \Cref{section:rigidity} we finish the proof of \Cref{theo:theorem1} and show
how it implies the rigidity stated in \Cref{theo:main_rigidity}.
Finally, in \Cref{section:covering_geodesic} we prove \Cref{lemma:almost_geodesic} (see also a remark in page 2 of \cite{KahnMark} for the totally marked case) thereby concluding the proof of \Cref{cor:covering_rigidity}.
\subsection{Acknowledgments} I am very grateful to my advisor Benson Farb for his guidance and constant support
throughout this project, and for his numerous comments on earlier drafts of the paper that greatly improved the exposition.
I would like to thank Curtis McMullen for comments on an earlier draft and suggesting
\Cref{question:wp,question:question1,question:question2}; and Alex Wright for many insightful conversations, including sharing with me \cite{Wright-AH}*{Question 10.6} which inspired this work. Finally, I am grateful to
Sidhanth Raman for our numerous conversations on this subject, for sharing with me his knowledge of deformation theory
and complex analytic geometry and for extensive comments on a previous draft of the paper. Finally, we thank the
anonymous referee for their helpful suggestions.

\section{Teichm\"uller geometry and the Weil--Petersson metric}\label{section:prelim}
\subsection{Basic definitions}Assume $3g-3 +n > 0$ and let
$\Sgn$ be an oriented topological surface of genus $g$ with $n$-marked points (or $n$-punctures).
A marking is an orientation preserving homeomorphism $\varphi:\Sgn \to X$ from $\Sgn$ to
a Riemann surface $X$.
Teichm\"uller space
$\Tgn$ is the space of equivalence classes of markings
$[ \varphi:\Sgn \to X ]$, where two markings are equivalent if they differ up to isotopy
by a biholomorphism. Unless we need to stress the marking we will denote $[\varphi:\Sgn \to X]$
simply by $X$. In case we need to emphasize the subset of points being marked we introduce the following notation:
Let $\Sigma$ be a topological surface and
$u \subset \Sigma$ be a finite subset of points. $\Tt(\Sigma,u)$ denotes
the Teichm\"uller space of $\Sigma$ with the marked points given by the subset $u$.

The mapping class group $\Mod(\Sgn)$---the group of orientation-preserving
diffeomorphisms of $\Sgn$ up to isotopy---acts on $\Tgn$ by biholomorphisms.
In fact, Royden and Earl--Kra~\cite{Royden,earle-kra} showed that for $2g +n >4$, the group
$\Mod(\Sgn)$ agrees with the group
of biholomorphisms of $\Tgn$. The pure mapping class group $\PMod(\Sgn)$
is the subgroup of $\Mod(\Sgn)$ fixing pointwise the marked points.
The quotient of $\Tgn$ by $\PMod(\Sgn)$ induces the projection
\[ \pi:\Tgn \to \Mgn = \Tgn/\PMod(\Sgn). \]

Let $B(X)$ denote the space of \emph{Beltrami forms} on $X$, i.e.
$(-1,1)$-forms $\mu$ on $X$ with $\|\mu\|_\infty < \infty$.
Denote by $M(X)$ the open unit ball in $B(X)$. There exists a holomorphic
submersion
\[ \Phi:M(X) \to \Tgn \ \ , \ \ \mu \to [f^{\mu}\circ \varphi:\Sgn \to X_\mu] \]
where $f^{\mu}$
is given by solving (locally) the Beltrami equation
\[ \partial_{\bar{z}}f = \mu \partial_z f. \]
In particular, $T_X\Tgn$ is naturally identified with a quotient of $B(X)$.

Let $[\varphi:(\Sigma,u) \to (X,\varphi(u))] \in \Tt(\Sigma,u)$. Let $Q(X,u)$
denote the space of integrable holomorphic quadratic differentials on $X$, i.e.
meromorphic quadratic differentials $q$ on $X$ with poles of order at most one, and all poles contained in $\varphi(u)$.
Equivalently, $Q(X,u)$ is the space of holomorphic quadratic differentials $q$ on $X-\varphi(u)$ so that
\[ \int_{X-\varphi(u)} |q| < \infty.\]
When the marked points are not being stressed we denote $Q(X,u)$ simply by $Q(X)$.
There is a natural pairing
\[B(X) \times Q(X) \to \CC \ \ , \ \ (\mu,q) \mapsto \int_X \mu q \]
which factors through $T_X\Tgn \times Q(X)$. In particular,
$Q(X)$ is naturally identified with $T^*_X \Tgn$.

At the level of the universal cover and working always with punctures, we have the following notation.\footnote{A similar description is possible
  for marked points by using groups with elliptic elements~\cite{earle-kra-sections}.}
Let $u \subset \Sigma$ be a finite set of points and $\varphi:(\Sigma ,u)\to (X,\varphi(u))$ be a marked Riemann surface.
Let $X' = X - \varphi(u)$ and let $\pi:\HH^2 \to X'$ be the universal cover of $X'$, with deck group $\Gamma$. Then $B(X')$ is given by the $\Gamma$-invariant
$(-1,1)$-forms on $\HH^2$, denoted by $B(\Gamma)$. Similarly, $Q(X')$ is given by $Q(\Gamma$), the $\Gamma$-invariant holomorphic
quadratic differentials on $\HH^2$
inducing an integrable quadratic differential on $X'$.

\subsection{Weil--Petersson metric}
The \emph{Weil--Petersson} metric, denoted by $g_{\WP}$, is induced from the hermitian product on $Q(X)$ given by
\[ \langle q_1,q_2\rangle_{\WP}= \int_X q_1 \overline{q_2} (ds^2)^{-1}, \]
where $ds^2$ is the hyperbolic volume form.
The hermitian product on $T\Tt(X)$
is given by
\[ \langle \mu_1,\mu_2\rangle_{\WP} = \overline{\langle \mu_1^*,\mu_2^*\rangle}_{\WP} \]
where $\mu_i^* \in Q(X)$ is the dual of $\mu_i$.
This means that
$\mu_i^*$ is characterized by
\[ \langle q,\mu_i^*\rangle_{\WP} = \int_X \mu_i q,  \ \  \forall q \in Q(X). \]
The metric $g_{\WP}$ is incomplete, K\"ahler, negatively curved
and geodesically convex~\cite{Wolpert}.

\medskip \noindent \textbf{Harmonic forms.} The map \[ q \mapsto \overline{q} (ds^2)^{-1}\]
sends a quadratic differential to an element of $B(X)$ which is called a
\emph{harmonic form}. The space of harmonic forms on $X$ is denoted by $H(X)$. The
Weil--Petersson metric has a simple expression on harmonic forms:
Let $\mu_i = \overline{q_i}(ds^2)^{-1} \in H(X)$ for $i=1,2$.
Evidently $\mu^*_i = q_i$, and so
\[ \langle \mu_1,\mu_2\rangle_{\WP} = \int_X \overline{q_1}q_2 (ds^2)^{-1} = \int_X \mu_1 \overline{\mu_2} ds^2 =
  \int_X \mu_1 q_2. \]

\begin{definition}[\textbf{Almost Weil--Petersson}]\label{def:almost_totally_geodesic}
  Given a complex submanifold $M \subset \Tgn$, we say
  that $M$ is \emph{Weil-Petersson} if it is totally geodesic with respect
  to $g_{\WP}$.\footnote{Note that unlike the Teichm\"uller totally geodesic case, it is not hard to see that
    any totally geodesic analytic subset of $\Tgn$ with respect to $g_{\WP}$ must in fact be a submanifold. Indeed, applying the exponential map we see that such a subset is a smooth submanifold and the claim follows (see~\cite{milnor}*{Remark on p 13}).}
  More generally, we say that $M$ is \emph{almost Weil--Petersson} if there exists a forgetful
  map $\mathcal F: \Tgn \to \Tgn[g,m]$ with $n\geq m$ and $3g-3+m > 0$. Such that
  \begin{enumerate}
    \item $\mathcal F|_M:M \to \mathcal F(M)$
      is a biholomorphism
      \item $\mathcal F(M)$
  is a Weil-Petersson complex submanifold of $\Tgn[g,m]$.
\end{enumerate}
  Let $N \subset \Mgn$ be a subvariety, a \emph{lift} of $N$ is an irreducible component
  of the preimage of $N$ in $\Tgn$. The subvariety $N$ is called \emph{(almost) Weil--Petersson} if
  it has an \emph{(almost) Weil--Petersson} lift.
\end{definition}

\section{$\Gamma_M$-deformations, Normalization and the associated orbifold fundamental group}\label{section:deformation}
Let $\pi: \Tgn \to \Mgn$ be the projection to moduli space. Given a subvariety $N \subset \Mgn$ with a lift $M \subset \Tgn$,
denote the stabilizer of $M$ in $\PMod(\Sgn)$ by $\Gamma_M$. In this section we introduce
the concept of a $\Gamma_M$-deformation which generalizes
an important property of classical deformations to the orbifold category.
$\Gamma_M$-deformations are central to our reformulation of \Cref{theo:main_rigidity},
stated at the end of this section (\Cref{theo:theorem1}). We start by proving
some facts about the group $\Gamma_M < \PMod(\Sgn)$.\\

\noindent \textbf{The associated orbifold fundamental group.}
The projection $\pi$
induces a holomorphic map from the analytic quotient\footnote{Note that when forming the quotient
  we remove the kernel of the action of $\Gamma_M$ on $M$.} $M/\Gamma_M$~\cite{Cartan}.
\[ \pi_M: M/\Gamma_M \to N. \]
The group $\Gamma_M$ is called the
\emph{associated orbifold fundamental group} of $N$(cf. ~\cite{Wright-AH}).
The following explains the terminology.

\begin{lemma}[cf. \cite{Diez-Harvey}*{Theorem 1}]\label{lemma:normalization}
  Suppose $3g - 3 + n > 0$. Given an analytic subvariety $N \subset \Mgn$ with lift $M \subset \Tgn$,
  let $\Gamma_M$ be the stabilizer of $M$ in $\PMod(\Sgn)$. The following hold:
  \begin{enumerate}
    \item The induced holomorphic map
  \[\pi_M:M/\Gamma_M \to N\]
  is proper with finite fibers and injective outside a proper analytic subvariety. In particular, $\pi(M) = N$ and so
  the conjugacy class of $\Gamma_M$ in $\PMod(\Sigma_{g,n})$ is an invariant of $N$.
\item If $N$ is quasi-projective, the normalization of $M/\Gamma_M$ is quasi-projective. If in addition $M$
  is normal, $M/\Gamma_M$ is a quasi-projective variety.
  \end{enumerate}
\end{lemma}
\begin{proof} \mbox{}\\
  \medskip \noindent \emph{Proof of item (1).}
  We need to show that the projection $\pi_M$ satisfies the following:
\begin{enumerate}
\item has finite fibers,
\item is closed,
\item is injective outside a proper analytic subvariety.
\end{enumerate}

The family of closed subsets $\{gM:g \in \PMod(\Sgn)\}$ in $\Tgn$ is a subfamily of the decomposition of $\pi^{-1}(N)$ into its irreducible components.
Thus, it is a locally finite family of closed subsets of $\Tgn$. This means
that for every $x \in \Tgn$, there exists a neighborhood $U_x$ of $x$
intersecting only finitely many $g_iM$. This automatically gives that $\pi_M$ has finite fibers, and an elementary argument shows that
$\pi_M$ is closed.

By the Semi-Proper Mapping Theorem~\cite{Fischer}*{Theorem 1.19}, the map $M \to M/\Gamma_M$
sends $\Gamma_M$-invariant analytic subsets of $M$ to analytic subsets of $M/\Gamma_M$.
The map $\pi_M$ is injective on the complement of the projection of the subset
\[ W := \bigcup_{g\in \Gamma_M} \{gM \cap M: g(M) \cap M \mbox{ proper subvariety of $M$} \} \subset M. \]
$W$ is analytic and $\Gamma_M$-invariant. Thus, the projection of $W$
to $M/\Gamma_M$ is analytic and the third item follows.
Since $\pi_M:M/\Gamma_M \to N$ is proper, the proper mapping theorem implies that
$\pi_M$ is surjective whenever $\dim(M/\Gamma_M) = \dim(N)$ and the last claim follows.
\medskip \noindent \emph{Proof of item (2).}
If $N$ is quasi-projective its normalization is also quasi-projective as the normalizations in the algebraic and analytic category coincide~\cite{Kuhlmann}*{Satz 4}.
Let $\xi:\hat{M} \to M/\Gamma_M$ be the normalization of $M/\Gamma_M$, then $\pi_M \circ \xi:\hat{M} \to N$ is the normalization of $N$ and so $\hat{M}$ is quasi-projective.
If $M$ is normal then $M/\Gamma_M$ is a normal analytic space by
a theorem of Cartan~\cite{Cartan}*{Theorem 4}, and the last claim follows.
\end{proof}
\begin{remark} As the examples in \cite{Diez-Harvey,Hidalgo_2024} show, $\pi_M$ is not in general a biholomorphism.
\end{remark}
\noindent\textbf{Level L-structures and smoothness of the normalization.}
Suppose $L \geq 3$ and $3g -3 +n > 0$, and let $\PMod(\Sgn)[L]$ be the
level $L$ pure mapping class group~\cite{farb-margalit}*{Chapter 6.4.2}.
Define \[\Gamma_M[L] := \Gamma_M \cap \PMod(\Sgn)[L]. \]
Since $L \geq 3$, $\PMod(\Sgn)[L]$ is torsion free. Thus, the induced map
\[ M \to M/\Gamma_M[L] \]
is an unramified covering map. Denote by
\[ \pi_L:\Tgn \to \mathcal M_{g,n}[L] \]
the level-$L$ projection to the moduli space of curves with level $L$
structures, and denote the image $\pi_L(M)$ by $N[L]$.
The same argument as in \Cref{lemma:normalization} shows
that whenever $N$ is
a subvariety of $\Mgn$, the normalization of $M/\Gamma_M[L]$ is a quasi-projective variety. In particular, if $M$ is smooth
then $M/\Gamma_M[L]$ is a smooth quasi-projective variety and the induced map
\[ \pi_{M[L]}: M/\Gamma_M[L] \to N[L] \subset \mathcal M_{g,n}[L] \]
defines the normalization of an irreducible component of the preimage of $N$ to $\mathcal M_{g,n}[L]$.
Observe that if $N[L]$ is normal, then $\pi_{M[L]}$ is a biholomorphism.

\begin{remark}[\textbf{On orbifold issues}]\label{rmk:orbi-issues} In this work, we will emphasize the orbifold category by explicitly lifting
  to a level $L$ cover of $\Mgn$ for $L \geq 3$. Fully acknowledging that this introduces more cumbersome
  notation, we suggest the reader to tacitly replace $\Mgn$ with $\mathcal M_{g,n}[L]$ and skip the discussion on
  orbifold structures on a first reading. In that case, all proofs and definitions then proceed by skipping the first step of lifting to an appropriate cover.
\end{remark}

\begin{definition}[\textbf{Good and algebraic Orbifolds}]\label{def:alg_orbi}
  Let $\widetilde{X}$ be a locally compact irreducible analytic space~\cite[Chapter 9]{GR}, and
  $\Gamma_X$ a group acting properly discontinuously on $\widetilde{X}$ by biholomorphisms,
   we denote the quotient by $X = \widetilde{X}/\Gamma$.
  Assume that there exists a finite index normal
  subgroup $\Gamma'_X \trianglelefteq \Gamma_X$ acting freely on $\widetilde{X}$.

  We call such structure a \emph{good orbifold} structure on $X$ and define $\pi_1^{\orb}(X) = \Gamma_X$.
  We say that the orbifold $X = \widetilde{X}/\Gamma_X$ with orbifold
  structure $(\widetilde{X},\Gamma_X)$ is \emph{algebraic} if the \emph{normalization} of $X':=\widetilde{X}/\Gamma'_X$ is a quasi-projective variety.

  Examples of algebraic orbifolds are given by $M/\Gamma_M$ for $M$ a lift of a subvariety of $\Mgn$, and
  by finite type curves $C$ with the orbifold structure $(\widetilde{C},\pi_1(C))$ for $\widetilde C$ the universal cover of $C$.

\end{definition}

\begin{definition}[\textbf{Proper and almost Weil--Petersson immersions}]\label{def:WP-im}
  Let $X$ be an orbifold with good orbifold structure $(\widetilde{X},\Gamma_X)$.
  Let $ f:X \to \Mgn$
  be a holomorphic map of complex orbifolds with associated homomorphism
  $f_*: \Gamma_X \to \PMod(\Sgn)$. This means
  that there is a holomorphic map $\tilde{f}:\widetilde{X} \to \Tgn$
  equivariant with respect to the homomorphism $f_*$ and fitting into the diagram
  \begin{center}
    \begin{tikzcd}
      \widetilde{X} \ar[r,"\tilde{f}"] \ar[d] &\Tgn \ar[d] \\
      X \ar[r,"f"'] &\Mgn
      \end{tikzcd}
    \end{center}
    The map $f$ is called a \emph{proper orbifold immersion} if it is proper, $\widetilde{X}$ is a
    complex manifold and $\tilde{f}$ is an immersion \emph{onto} an analytic subvariety $M \subset \Tgn$. In particular, the image $N:=f(X)$ is an \emph{analytic} subvariety of
    $\Mgn$, and $M$ is a lift of $N$. If
    in addition the image $N$ is an \emph{almost Weil--Petersson} subvariety, $f$ is called
    an \emph{almost Weil--Petersson immersion}. Note that in this case $M$ is an almost Weil-Petersson submanifold and $\tilde{f}$ is an embedding.
    As consequences of both definitions we find
  \begin{enumerate}
  \item $X$ is a normal analytic space.
  \item The homomorphism $f_*$ factors through the stabilizer $\Gamma_M$ of $M$.
  \item For any $L \geq3$, there exists a finite index normal subgroup $\Gamma_X[L]$ of $\Gamma_X$ with manifold quotient
    $X[L] = \widetilde{X}/\Gamma_X[L]$ such that
    the induced map
    \[ X[L] \to M/\Gamma_M[L] \to N[L] \subset \mathcal{M}_{g,n}[L] \]
    is a proper immersion.
  \end{enumerate}
  For an almost Weil--Petersson immersion the following is also satisfied.
  \begin{enumerate}
  \item $X[L]$ is a finite unbranched cover of $M/\Gamma_M[L]$, thus $X[L]$ is a smooth quasi-projective variety
      and $(\widetilde{X},\Gamma_X)$ is an algebraic orbifold structure on $X$. Furthermore,
      since $X$ is the quotient of $X[L]$ by a finite group, $X$ also inherits a quasi-projective structure.
  \end{enumerate}

  Examples of almost Weil--Petersson immersions are given by the orbifold maps $M/\Gamma_M \to \Mgn$
  for $M$ a lift of an almost Weil--Petersson subvariety. As \Cref{lemma:normalization} shows, this is
  precisely the normalization map of an almost Weil--Petersson subvariety.
\end{definition}

\noindent\textbf{Classical deformations.} Let $Z,W$ be two complex manifolds. A
local deformation of a holomorphic map $f:Z \to W$ is given by the following data
$(\mathcal Z, \mathcal F, Z,\varphi, \pi)$ where:
\begin{enumerate}
  \item $\mathcal F : \mathcal Z \to W$
    is a holomorphic map from a complex manifold $\mathcal Z$.
    \item
      $\pi: \mathcal Z \to \Delta$ is a holomorphic submersion defining a \emph{family} $Z_t:= \pi^{-1}(t)$
      and $f_t := \mathcal F|_{Z_t}$.
    \item $\varphi:Z \to Z_0 = \pi^{-1}(0)$ is a biholomorphism with the property that
$\mathcal F \circ \varphi = f$.
\end{enumerate}
 If the family $\mathcal Z$ is smoothly trivial,\footnote{note that we are \emph{not assuming} that $\pi$ is proper.} it follows that the groups $(f_t)_*(\pi_1(Z_t)) < \pi_1(W)$ are
pairwise conjugate. The following is our generalization to the orbifold setting
of a deformation of the map $\pi_M:M/\Gamma_M \to N \subset \Mgn$.

\begin{definition}[\textbf{$\Gamma$-deformation}]\label{defn:gamma}
  Let $X= \widetilde{X}/\Gamma_X$ be an algebraic orbifold.
  Let $ \phi:X \to \Mgn$
  be a holomorphic map of complex orbifolds with associated homomorphism
  $\phi_*: \Gamma_X \to \PMod(\Sgn)$. Given an arbitrary subgroup $\Gamma < \PMod(\Sgn)$,
  we say that $\phi$ is a \emph{$\Gamma$-deformation} if, up to conjugation in $\PMod(\Sgn)$,
  the homomorphism $\phi_*$ factors through the inclusion $\Gamma \hookrightarrow \PMod(\Sgn)$.
\end{definition}
Our main rigidity result is the following variant of the Imayoshi-Shiga theorem (see also \Cref{theo:Imayoshi-Shiga}).
\begin{theorem}\label{theo:theorem1}
Let $3g-3+n > 0$. Let $M$ be an almost Weil--Petersson
  complex submanifold of $\Tgn$, and $\Gamma_M < \PMod(\Sgn)$
  be the stabilizer of $M$. Given an algebraic orbifold  $X=\widetilde{X}/\Gamma_X$ and
  a nonconstan $\Gamma_M$-deformation $\phi:X \to \Mgn$. Then, there is a lift
  $\tilde{\phi}:\widetilde{X} \to \Tgn$ of $\phi$ that factors through the inclusion $M \hookrightarrow \Tgn$.
\end{theorem}

\section{The case of a finite type curve $C$}\label{section:main_theo}

On this section we prove \Cref{theo:theorem1} when $X$ is a smooth finite type
curve. We start with the following special case.

\begin{proposition}\label{lemma:rigidity_curves} Suppose $3g-3+n > 0$. Given
  a Weil--Petersson complex submanifold $M \subset \Tgn$
  with stabilizer $\Gamma_M<\PMod(\Sgn)$. Let $C$
  be a smooth finite type curve and $\phi:C \to \Mgn$ be a nonconstant $\Gamma_M$-deformation.
  Then, there exists a lift of $\phi$ to $\Tgn$ that factors through the inclusion
  $M \hookrightarrow \Tgn$.
\end{proposition}

The first step in the proof of \Cref{lemma:rigidity_curves} is the following
elementary observation,
whose proof we include for completeness.

\begin{lemma}[\textbf{Orthogonal Projection}]\label{lemma:projection} Suppose
  $3g - 3 + n > 0$. Let $Z$ be
  a totally geodesic submanifold of $\Tgn$ with respect to the Weil--Petersson metric and
  let $\Gamma_Z$ be the stabilizer of $Z$ in $\Mod(\Sgn)$. There exists a
  $\Gamma_Z$-equivariant smooth deformation retraction onto $Z$
  \[ \pi^t_Z:\Tgn \to \Tgn  \ \ , \ \ t \in [0,1]\]
  with $\pi^0_Z = 1$ and $\pi^1_Z$ the retraction onto $Z$. Such that the following
  holds: for any $p \in \Tgn$ and $w \in T_p\Tgn$ the map
  \[t \to \|d\pi^t_Z(w)\|_{\WP} \]
  is non-increasing. Moreover, say $p \notin Z$ and let $\gamma_P$ be the geodesic joining $p$ and $\pi^1_Z(p) \in Z$. If $w \notin T_p\gamma_p$ then $\|d\pi^t_Z(w)\|_{\WP}$ is strictly decreasing.
\end{lemma}
\begin{proof}
  Endow $\Tgn$ with the Weil--Petersson metric $g_{\WP}$. Let $\Ee \subset T\Tgn$ be the domain of the exponential map. Since $g_{\WP}$ is incomplete $\Ee \not= T\Tgn$.
  Let $\Ee_Z:= \Nn Z \cap \Ee$ for $\Nn Z$ the normal bundle of $Z$.
  The metric $g_{\WP}$ is negatively curved, thus
  \[ \exp:\Ee_Z \to \Tgn \]
  is a local diffeomorphism onto its image.
  Since $g_{\WP}$ is not complete, a priori there is no reason for $\exp$ to be bijective.
  Yet, Wolpert showed that for any $p \in \Tgn$, the map $\exp_p:\Ee_Z \cap T_p\Tgn \to \Tgn$
  is a homeomorphism~\cite{Wolpert}*{Corollary 5.4}. Coupling this fact with the fact
  that distance in $g_{\WP}$ is measured along geodesics~\cite{Wolpert}*{Corollary 5.9} implies the surjectivity of $\exp$. Similarly, these two facts also imply that geodesic triangles have total internal angle $\leq \pi$ and from this it follows that $\exp$ is injective.

  Let $\phi_t:\Ee_Z \to \Ee_Z$ be given by $(p,v) \to (p,tv)$. Define
  \[ \pi^t_Z:\Tgn \to \Tgn  \ \ , \ \  \pi^t_Z:=\exp \circ \phi_{1-t} \circ \exp^{-1}.\]
  Since $\exp$ and $\phi_t$ are $\Gamma_Z$-equivariant, it follows that $\pi^t_Z$ is
  $\Gamma_Z$-equivariant. The map $\pi^1_Z$ is precisely the orthogonal projection map to $Z$.

  Let $p \in \Tgn$ and $w \in T_p \Tgn$. Observe that
  $J(t) = d\pi^{1-t}_Z(w)$ is a Jacobi field along the geodesic $\gamma_p(t)= \pi^{1-t}_z(p)$.
  A standard computation,
  using the fact that $Z$ is totally geodesic, shows that $\|J(t)\|_{\WP}$
  is non-decreasing. Indeed, recall that negative curvature implies that $\|J(t)\|_{\WP}^2$ is
  \emph{convex} so we only need to show that $\frac{d}{dt}|_0 \|J(t)\|_{\WP}^2 \geq 0$. This
  is achieved by observing that the variation associated to $J(t)$ is given, via $\exp$,
  through geodesics perpendicular to a geodesic $\gamma' \subset Z$.

  Furthermore, if $p \notin Z$ and $w$ is not parallel to $\gamma_p'(1)$ it
  follows that $J(t)$ has a non-trivial normal component
  and so $\|J(t)\|_{\WP}$ is strictly increasing. \end{proof}

Having \Cref{lemma:projection} at hand, we can prove \Cref{lemma:rigidity_curves}.
\begin{proof}[Proof of \Cref{lemma:rigidity_curves}]
  By \Cref{lemma:projection} there is a smooth $\Gamma_M$-equivariant homotopy
  \[\pi^t_M:\Tgn \to \Tgn \ \ , \ \ t \in [0,1]\]
  between the identity and the orthogonal projection $\pi^1_M:\Tgn \to M$.
  Given a a nonconstant $\Gamma_M$-deformation
  $\phi:C \to \Mgn$, it follows that $C$ is of hyperbolic type.
  In particular, the associated lift $\tilde{\phi}:\HH^2 \to \Tgn$ is \emph{not constant}. Up to post-composing
  with a mapping class, we can assume that $\phi_*:\pi_1(C) \to \Gamma_M.$
  Define the map
\[ \tilde{\psi}:\HH^2 \to \Tgn \ \ , \ \ \tilde{\psi} = \pi^1_M \circ \tilde{\phi}. \]
Since $\pi^1_M$ is $\Gamma_M$-equivariant, it follows that $\tilde{\psi}$ is $\phi_*$-equivariant. \Cref{lemma:rigidity_curves} will be proven
if we show that $\tilde{\phi} = \tilde{\psi}$.
\begin{claim}\label{claim:c} $\tilde{\phi} = \tilde{\psi}$
\end{claim}
\begin{proof}[Proof of \Cref{claim:c}]
The proof follows the argument appearing in the proof of the Imayoshi--Shiga theorem given in Section 4 of \cite{AAS}.
We include the details for completeness,
clarifying the equivariance nature of all the maps being used---
hence the difference in our notation.

For a smooth map $\tilde{f}:\HH^2 \to \Tgn$ define the \emph{energy density} $E_x(\tilde{f})$ of $\tilde{f}$ at a point $x \in \HH^2$
to be
\[ E_x(\tilde{f}) := \| d\tilde{f}|_x(v_1) \|_{\WP}^2 + \| d\tilde{f}|_x(v_2) \|_{\WP}^2,\]
for an arbitrary orthonormal basis $\{v_1, v_2\}$ of $T_x\HH^2$. Assume further
that $\tilde{f}$ is $\phi_*$-equivariant, then $E_x(\tilde{f})$ is $\pi_1(C)$-invariant and
we can define the \emph{equivariant energy} of $\tilde{f}$ as:
\[ E(f) := \int_C E_{\bar{x}}(f)\omega_C \]
where $\omega_C$ is the volume form on $C$ and $E_{\bar{x}}(f)$ is given by $E_{x}(\tilde{f})$ for any lift $x$ of $\bar{x}$.

Since $\tilde{\phi}$ is holomorphic and $d_{\Teich}$ agrees with
the Kobayashi metric~\cite{Royden}, it follows that $ \|d\tilde{\phi}(v)\|_{\Teich} \leq \|v \|_{\HH^2}$.
Recall that the Teichm\"uller metric dominates the Weil--Petersson metric
  \[ \|v\|_{\WP} \leq L \|v\|_{\Teich}, \]
  where $L = |2\pi(2g-2 + n)|^{1/2}$~\cite{McMullen-Kahler}*{Proposition 2.4}. Thus, $\|d\tilde{\phi}(v) \|_{\WP} \leq L\|v\|_{\HH^2}$.
  The norm non-increasing nature of $\pi^t_M$ implies the following

\begin{lemma}[cf. \cite{AAS}*{Lemma 4.1}]\label{lemma:ortho_props} The
  homotopy
  \[ \widetilde{F}:\HH^2 \times [0,1] \to \Tgn, \ \ \tilde{f}_t(x) = \pi^t_M(\tilde{\phi}(x)) \]
  between $\tilde{\phi}$ and $\tilde{\psi}$ satisfies:
  \begin{enumerate}
  \item $\tilde{f}_t$ is $L$-Lipschitz for all $t$. Thus, the equivariant energy
    $E(f_t)$ is finite for all $t$.
  \item The map $t \to E(f_t)$ is non-increasing along $t$. If $\tilde{\phi} \not = \tilde{\psi}$, then $E(f_t)$ is not constant.
  \item For any two points $\bar{x},\bar{y} \in C$ let $d_C(\bar{x},\bar{y})$ be
    their distance in $C$.
    Fix $\bar{x}_0 \in C$,
    there are constants $A,B > 0$ such that for all $(x,t) \in \HH^2 \times [0,1]$
    the operator norm $\|d\widetilde{F}_{(x,t)}\|$ satisfies
    \[ \|d\widetilde{F}_{(x,t)} \|^2 \leq A d_C(\bar{x}_0,\bar{x})^2 + B.\]
  \end{enumerate}
\end{lemma}
\begin{proof}
  The first two items are an immediate consequence of \Cref{lemma:projection}.
  For the last one, observe that $\widetilde{F}$ gives the geodesic homotopy between $\tilde{\psi}$ and $\tilde{\phi}$. Thus,
  the same arguments as in \cite{AAS}*{Lemma 4.1} apply---replacing
  the holomorphicity of $\tilde{\psi}$ with the fact that
  $\tilde{\psi}$ is $L$-Lipschitz.
\end{proof}
\begin{remark}
  Note that since the homotopy $\pi^t_M$ is norm non-increasing we get
  stronger results than \cite{AAS}*{Lemma 4.1}.
\end{remark}

The proof of \Cref{claim:c} is completed by the two following key results. The first of which is the following elementary observation, a consequence of the general formula for the energy density ~\cite{ES}*{page 8}. Let $\omega_{\WP}$ be
the K\"ahler form associated to the Weil--Petersson metric. Similarly, let $\omega_{\HH^2}$ be the K\"ahler form associated to the Poincar\'e metric $g_{\HH^2}$.
\begin{lemma}[cf. \cite{AAS}*{Proposition 4.2}]\label{prop:Wirtinger} Let $f:\HH^2 \to (\Tgn,g_{\WP})$ be
a smooth map. Then, for any $x \in \HH^2$
\[E_x(f) \geq \frac{f^*(\omega_{\WP})}{\omega_{\HH^2}}|_x\]
with equality if $f$ is holomorphic at $x$.
\end{lemma}

\begin{remark}
  \cite{AAS}*{Proposition 4.2} gives an if and only if condition for equality, based on \cite{ES}*{Proposition 4.2}; this is not necessary for our proof.
\end{remark}
For each $t$, let $f_t^*\omega_{\WP}$ be the \emph{closed} $2$-form on $C$ induced by the $\pi_1(C)$-invariant form
$\tilde{f}_t^*\omega_{\WP}$ on $\HH^2$. We have the following,

\begin{proposition}[\cite{AAS}*{Proof of Imayoshi--Shiga}]\label{prop:Stokes}
For all $t \in [0,1]$ we find
\[ \int_C f_t^*(\omega_{\WP}) = \int_C f_0^*(\omega_{\WP}). \]
\end{proposition}
\begin{proof}
  Since $\widetilde{F}$ satisfies properties (1)-(3) of \Cref{lemma:ortho_props}, the same computation as
  in the proof of the Imayoshi--Shiga theorem in \cite{AAS} gives the result. This is
    a generalization of Stoke's theorem by using a nice enough exhaustion of $C$ by compact sets.
    We remark that the notation of \cite{AAS} differ from ours as all of their maps
    should be understood in the orbifold sense, e.g. $\|dF_{(x,t)}\|$ is
    given in our notation by $\|d\widetilde{F}_{(\tilde{x},t)}\|$ for any lift $\tilde{x}$
    of $x$.
\end{proof}
\noindent\textbf{Finishing the proof of $\tilde{\phi} = \tilde{\psi}$.} On one hand via \cref{prop:Wirtinger}
\[ E(f_t) \geq \int_C f_t^*(\omega_{\WP}).\]
On the other hand via \cref{prop:Stokes} and the fact that $f_0$ is holomorphic.
\[ \int_C f_t^*(\omega_{\WP}) = \int_C f_0^*(\omega_{\WP}) = E(f_0)\]
Thus $E(f_t) \geq E(f_0)$ and via property (2) in
\Cref{lemma:ortho_props} we find $E(f_t) = E(f_0)$ for all $t$. In particular $\tilde{\phi} = \tilde{\psi}$.
\end{proof}
Having proven \Cref{claim:c}, \Cref{lemma:rigidity_curves} follows.
\end{proof}

\begin{remark}
  Note that the same argument works more generally for $M$
  a real totally geodesic submanifold of $\Tgn$
  with respect to the Weil--Petersson metric, e.g. an axis of
  a pseudo-Anosov element of $\Mod(\Sgn)$~\cite{DW}.
  Furthermore, the same yields the following strengthening
  of the Imayoshi--Shiga theorem.
\end{remark}

\begin{theorem}[\textbf{Imayoshi--Shiga}]\label{theo:Imayoshi-Shiga}
  Suppose $3g - 3 + n > 0$. Let $C = \HH^2/\pi_1(C)$ be a smooth finite type curve. Given
  $i=1,2$ and smooth maps $\phi_i:\HH^2 \to \Tgn$, equivariant with respect to the same homomorphism
  $\phi_*:\pi_1(C) \to \PMod(\Sgn)$. Assume that $\phi_1$ is holomorphic
  and nonconstant. Denote the equivariant energy
  of $\phi_i$ by $E(\phi_i)$.
  If either $\phi_2$ is holomorphic or $\WP$-Lipschitz with
  \[ E(\phi_2) \leq E(\phi_1) \]
  then $\phi_1 = \phi_2$.
\end{theorem}

\Cref{lemma:rigidity_curves} implies the following special
case of \Cref{theo:theorem1} when the algebraic orbifold is a smooth finite type
curve.

\begin{proposition}\label{prop:rig_almostgeo}
  Let $3g -3 + n > 0$. Given an almost Weil--Petersson complex submanifold $M \subset \Tgn$
  with stabilizer $\Gamma_M<\PMod(\Sgn)$. Let $C$ be a smooth finite type curve and
  $\phi:C \to \Mgn$
  be a nonconstant $\Gamma_M$-deformation.
  Then, $\phi$ has a lift $\tilde{\phi}:\HH^2 \to \Tgn$ factoring through $M \hookrightarrow \Tgn$.
\end{proposition}

\begin{proof}
  Let $\phi:C \to \Mgn$ be a nonconstant $\Gamma_M$-deformation. Then, there is a lift
  $\tilde{\phi}:\HH^2 \to \Tgn[g,n]$ of $\phi$, equivariant
  with respect to a homomorphism
  \[ \phi_*:\pi_1(C) \to \Gamma_M \hookrightarrow \PMod(\Sgn).\]
  Since $M$ is almost Weil--Petersson, there exists a forgetful map
  \[ \mathcal F:\Tgn \to \Tgn[g,m], \]
  with $n\geq m$ and $3g -3 + m > 0$, so that $M' = \mathcal F(M)$ is a Weil--Petersson complex submanifold of $\Tgn[g,m]$.
  Let
  \[ \mathcal F_*:\Mod(\Sgn) \to \Mod(\Sgn[g,m])\]
  be the associated forgetful map. The map $\mathcal F$
  is $\mathcal F_*$-equivariant and
  evidently $\mathcal F_*(\Gamma_M) \subset \Gamma_{M'}$.
  We claim that
  \[ \mathcal F_*|_{\Gamma_M}:\Gamma_M \to \Gamma_{M'} \]
  is injective. Indeed, $\mathcal F|_M: M \to M'$ is a biholomorphism
  and so the only way for an element $\gamma \in \ker \mathcal F_*$ to be in $\Gamma_M$
  is to fix $M$ \emph{pointwise}. The kernel $\ker \mathcal F_*$
  is a surface braid group $B_{n-m}(\Sgn[g,m])$~\cite{farb-margalit}*{Theorem 9.1}.
  Surface braid groups for surfaces with negative Euler characteristic are torsion free~\cite{FN}*{Corollary 2.2}.
  As $\chi(\Sgn[g,m]) <0$
  the claim follows.

  The map $\tilde{\phi}$ is nonconstant, thus $\phi_*$ has infinite image. In particular, the group $G:=\mathcal F_*(\phi_*(\pi_1(C)))$ is infinite.
  It follows that the map $\bar{\phi}:= \mathcal F \circ \tilde{\phi}$ is not constant: indeed,
  if not then $G$ fixes a point $y$, but this is impossible
  as $G$ is infinite.

  Observe that the map $\bar{\phi}:\HH^2 \to \Tgn$ is $\mathcal F_* \circ \phi_*$-equivariant
  so via (the proof of) \Cref{lemma:rigidity_curves} we find that $\bar{\phi}$
  factors through $M' \hookrightarrow \Tgn[g,m]$. To conclude the proof of \Cref{prop:rig_almostgeo} we show that the holomorphic map
  \[ \hat{\phi}:= \mathcal F|_{M}^{-1} \circ \bar{\phi}:\HH^2 \to M \subset \Tgn \] equals $\tilde{\phi}$.
  By \Cref{theo:Imayoshi-Shiga}, it is enough
  to show that $\hat{\phi}$ is $\phi_*$-equivariant. Let $x \in M'$ and $\gamma \in \Gamma_M$, then
  \[ \mathcal F|_{M}^{-1}(\mathcal F_*(\gamma)\cdot x) = \mathcal F|_{M}^{-1}(\mathcal F_M(\gamma\cdot\mathcal F|_{M}^{-1}(x)) = \gamma\cdot\mathcal F|_{M}^{-1}(x).\]
  Thus $\hat{\phi}$ is $\phi_*$-equivariant and the claim follows.
\end{proof}

\section{Rigidity}\label{section:rigidity}

In this section we complete the proofs of \Cref{theo:theorem1} and \Cref{theo:main_rigidity}.

\begin{proof}[Proof of \Cref{theo:theorem1}]
  Let $X = \widetilde{X}/\Gamma_X$ be an algebraic orbifold with
  cover $X' = \widetilde{X}/\Gamma'_X$ having
  quasi-projective normalization $\hat{X}$.
  Given a nonconstant $\Gamma_M$-deformation $\phi:X \to \Mgn$,
  let $\phi':X' \to \Mgn$ be the induced orbifold map.
  By assumption there is a lift  $\tilde{\phi}:\widetilde{X} \to \Tgn$ of $\phi$ that is equivariant with respect to the homomorphism
  \[ \phi_*:\Gamma_X \to \Gamma_M \hookrightarrow \PMod(\Sgn).\]
  Let $C \subset X'$ be a smooth finite type curve so that
  $\phi'|_C$ is not constant. Via the proof of \Cref{prop:rig_almostgeo}, the induced map $\tilde{\phi}|_{\widetilde{C}}$ on the universal cover $\widetilde{C}$ factors through $M$.
  Thus, \Cref{theo:theorem1} will be proven if we show that we can cover $X'$
  by finite type curves $C$ such that $\phi'|_C$ is not constant. Note that $C$
  is not required to be smooth. This is because $\phi'|_C$ will induce
  a map from its normalization.

  \begin{lemma}\label{lemma:crap} Let $y \in X'$. There exists a finite type curve $C \subset X'$
    passing through $y$ so that $\phi'|_C$ is not constant.
  \end{lemma}
  \begin{proof}
    This is elementary, but we include the details. Given $y \in X'$, let $Y = (\phi')^{-1}(\phi'(y))$. Then, $Y$ is a proper non-empty analytic subset of $X'$.
    Let $\xi:\hat{X}\to X'$ be the normalization of $X'$ with ramification locus $R$. Since $\hat{X}$ is quasi-projective
    there exists a finite type curve $C$ (not necessarily smooth) passing through any two of its points.
    Pick a point $x \in \hat{X} \setminus \xi^{-1}(Y) \cup R$, and let $C$ be a finite type curve passing through $x$ and an arbitrary preimage $y'$ of $y$.
    Then $C_y:=\xi(C)$ is a finite type curve passing through $y$. Indeed, the last claim follows since $C$ and $C_y$ have the same normalization.
\end{proof}
Having shown \Cref{lemma:crap}, \Cref{theo:theorem1} follows.
\end{proof}

Now we show that \Cref{theo:theorem1} implies
the rigidity stated in \Cref{theo:main_rigidity}. We consider the following
more general class of local deformations,
\begin{definition}\label{def:proper_def}
  Let $f:X \to \Mgn$ be a proper orbifold immersion (see \Cref{def:WP-im}) from
  a good orbifold $X = \widetilde{X}/\Gamma_X$ and $L \geq 3$.
  Let $\Gamma_X[L]$ be any finite index subgroup of $\Gamma_X$ acting
  freely on $\widetilde{X}$ and so that $f_*(\Gamma_X[L]) < \PMod(\Sgn)[L]$.
  A \emph{proper} local orbifold
  deformation of $f$ is a classical local holomorphic deformation
  \[ F_t: X_t[L] \to \mathcal M_{g,n}[L] \ , \ t \in \Delta \]
  of the lift $f[L]:X[L]:=\widetilde{X}/\Gamma_X[L] \to \mathcal M_{g,n}[L]$ of $f$.
  In addition, we require the following to hold:\footnote{Observe that these
    are automatic for local deformations of compact manifolds.}
  \begin{enumerate}
  \item  Let $\pi:\mathcal X[L] \to \Delta$ be the associated deformation
    of $X[L]$, with total space the complex manifold $\mathcal X[L]$. Then $\pi$ is a closed, smooth submersion.
  \item $F_t:X_t[L] \to \mathcal M_{g,n}[L]$ is proper for each $t \in \Delta$.
  \item $\mathcal X[L]$ is smoothly isomorphic to $X[L] \times \Delta$.
  \end{enumerate}
\end{definition}

\begin{theorem}\label{theo:deformation}
  Assume $3g -3 + n > 0$, $L\geq 3$ and $\dim(X) \geq 1$. Be given an almost Weil--Petersson immersion $f:X \to \Mgn$.
  Let
  \[ F_t: X_t[L] \to \mathcal M_{g,n}[L]\ \ ,  \ \  t \in \Delta\]
  be a proper local orbifold deformation of $f$
  such that each member of the family $\pi:\mathcal X[L] \to \Delta$
  is quasi-projective. Then, up to shrinking $\Delta$, the deformation is trivial,
  i.e. there exists a holomorphic family of biholomorphisms $g_t:X_t[L] \to X[L]$ inducing $(X_t[L],F_t) \cong (X[L], f[L])$.
  If each $F_t$ is an immersion, the result holds without shrinking $\Delta$.
\end{theorem}

\begin{proof}
  Let $M$ be a lift of the image $N:=f(X)$ and $\Gamma_M$ be the stabilizer of $M$ in $\PMod(\Sgn)$.
Denote by $\mathcal F:\mathcal X[L] \to \mathcal M_{g,n}[L]$ the holomorphic map associated to the
  deformation of $f$.

  The set of points $(x,t) \in \mathcal X[L]$ where $dF_t|_x$ is not full rank is closed and does not include the central fiber.
  Since $\pi:\mathcal X[L] \to \Delta$ is closed, by shrinking $\Delta$, we can assume that $F_t$ is an immersion for all $t$.
  By assumption $X_t[L]$ is quasi-projective
  for all $t$,  thus
  \begin{center}
  \begin{tikzcd}[sep=small]
    X_t[L] \ar[r, "F_t"] &\mathcal M_{g,n}[L] \ar[r] & \Mgn
  \end{tikzcd}
  \end{center}
  is a non-constant $\Gamma_{M}$-deformation. Since $M$ is almost Weil-Petersson, \Cref{theo:theorem1} implies that $F_t(X_t[L]) \subset N[L]=F_0(X_0[L])$ for all $t$. In particular,
  the map $\mathcal F:\mathcal X[L] \to \Mgn$ factors through $N[L]$. Recall that $\pi_{M[L]}:M/\Gamma_M[L]\to N[L]$ gives
  the normalization of $N[L]$. Hence, since $\mathcal F$ is surjective and
  $\mathcal X[L]$ is normal and irreducible,
  there is a unique holomorphic map $\mathcal G:\mathcal X[L] \to M/\Gamma_M[L]$ making the following diagram commute ~\cite{GR}*{Section 8.4.3}
\begin{center}
  \begin{tikzcd}[sep =small]
    & M/\Gamma_M[L] \ar[d,"\pi_{M{[}L{]}}"] \\
    \mathcal X[L] \ar[ur,"\mathcal G"] \ar[r,"\mathcal F"'] & N[L].
  \end{tikzcd}
\end{center}
 Let $\phi:X[L] \to X_0[L] \subset \mathcal X[L]$ be the biholomorphism so that
 $F_0\circ \phi = f[L]:X[L] \to N[L]$. Observe that, by definition,
 the surjective map $f[L]:X[L] \to N[L]$ factors through a finite holomorphic covering map
 \[ \hat{f}[L]:X[L] \to M/\Gamma_M[L].\]
 By uniqueness of lifts to the normalization, it follows that $\mathcal G \circ \phi = \hat{f}[L]$.
 In particular, $\mathcal G_*(\pi_1(\mathcal X[L])) = \hat{f}[L]_*(\pi_1(X[L]))$ and so we have
 a further lift
 \begin{center}
   \begin{tikzcd}[sep = small]
     & X[L] \ar[d, "\hat{f}{[}L{]}"] \\
     \mathcal X[L] \ar[r,"\mathcal G"'] \ar[ur,"\Psi"] & M/\Gamma_M[L].
   \end{tikzcd}
 \end{center}
 Denote by $\Psi_t$ the restriction of $\Psi$ to each fiber $X_t[L]$.
 Choosing $\Psi$ such that $\Psi\circ \phi:X[L] \to X[L]$ fixes a point, it follows that $\Psi \circ \phi$ is the
 identity, and so $\Psi_0= \phi^{-1}$ is a biholomorphism onto $X[L]$. Furthermore, by dimension reasons
 it follows that $\Psi_t$ is a proper local biholomorphism, i.e.
  a finite holomorphic covering map. Let $\varphi_t:X_t[L] \to X[L]$  for $t \in \Delta$, be a smooth
  trivialization. It follows that for each $t$, $\Psi_t \circ \varphi_t^{-1}:X[L] \to X[L]$ is a finite covering
  map homotopic to a homeomorphism. In particular, it has degree $1$. So $\Psi_t$ is a biholomorphism and
  $(X_t[L],F_t) \cong (X[L], f[L])$ via the family $\Psi_t$ and the claim follows.
\end{proof}

\Cref{theo:main_rigidity} follows from \Cref{theo:deformation}.

\begin{remark}[One dimensional case]
  Observe that the assumptions on $F_t:X_t[L] \to \mathcal M_{g,n}[L]$ being a proper immersion are only used at the end of the proof of \Cref{theo:deformation},
  to conclude that the induced maps to the normalization are covering maps. In the one dimensional case we can drop these assumptions. In fact, we have a stronger result.
\end{remark}
  \begin{theorem}\label{theo:rigidity-curves}
    Let $N \subset \Mgn$
    be an almost Weil--Petersson curve. Given a lift $M \subset \Tgn$ of $N$, suppose that
    $\chi(M/\Gamma_M) < 0$. Then, the normalization map $\pi_M: M/\Gamma_M \to \Mgn$
    is the \emph{unique} (in the analytic category) non-constant $\Gamma_M$-deformation $f:X \to \Mgn$ from an algebraic orbifold $X = \widetilde{X}/\Gamma_X$
    with $X$ a smooth finite type curve homeomorphic to $M/\Gamma_M$.
  \end{theorem}

  \begin{proof}
    By \Cref{theo:theorem1} any non-constant $\Gamma_M$-deformation $f:X  \to \Mgn$ induces a holomorphic map $h:X \to M/\Gamma_M$ so that $\pi_M \circ h = f$. The theorem would be proven if we show that $h$ is a biholomorphism.

    \medskip \noindent \emph{Step 1: Extension and properness.} Since $\chi(X) = \chi(M/\Gamma_M) < 0$, the map $h$ extends (see e.g. ~\cite[Prop 2.3]{Chen-Salter})
    to a holomorphic map $\bar{h}:\bar{X} \to \overline{M/\Gamma_M}$ of the
    compactified curves. Note that $\bar{h}$ is surjective and $X$ and $M/\Gamma_M$ have the
    same number of punctures. Thus, it follows that $\bar{h}$ maps the punctures of $X$ bijectively onto the punctures of $M/\Gamma_M$. In particular, $h$ is proper.

    \medskip \noindent \emph{Step 2: Degree one.} Let $k$ be the degree of $\bar{h}$ and let $\ell$ the number of punctures of $X$.
    The previous step shows that $\bar{h}$ is totally ramified at the punctures so by Riemann-Hurwitz
    \[ -\chi(\bar{X}) = -k\chi(\bar{X}) + \ell (k-1) + R ,\]
    for some $R \geq 0$. If $\chi(\bar{X})< 0$, this automatically implies that
    $k =1$. Similarly, for $\chi(\bar{X}) =0$ we must have $\ell \geq 1$ and so again
    $k=1$. Finally, for $\chi(\bar{X})=2$ we must have $\ell \geq 3$ and again
    this implies that $k = 1$. It follows that $h$ is a biholomorphism.
  \end{proof}

  \begin{remark}Note that the assumption on $\chi(M/\Gamma_M)<0$ can be always achieved by passing to a suitable finite cover. In particular, the uniqueness result holds for the map $M/\Gamma_M[L] \to \mathcal M_{g,n}[L]$ for $L \geq 3$.
    \end{remark}

\section{Covering constructions are Weil--Petersson geodesic}\label{section:covering_geodesic}
The main goal of this section is the proof of \Cref{lemma:almost_geodesic} relating
covering constructions with almost Weil-Petersson complex submanifolds. Note that\footnote{More classically, the same follows from the fact that a regular totally marked covering construction $f_h:\Tgn \to \Tgn[g',n']$ induces a proper algebraic map $\widetilde{\Mgn} \to \Mgn[g',n']$ from a suitable cover of $\Mgn$ (cf.~\cite{ACG}*{Chapter XIV}).} by a result of Filip (cf. ~\cite{filip}*{Remark 1.6})
covering constructions project to subvarieties of $\Mgn$, so we only need to show that
the images of covering constructions are almost Weil--Petersson complex
submanifolds of $\Tgn$.

We start by recalling the definition of covering constructions.
Let
\[h:\Sigma' \to \Sigma\] be a finite degree, orientation-preserving, topological branched cover between topological surfaces $\Sigma'$
and $\Sigma$. Consider finite subsets $u \subset \Sigma$ and $v\subset \Sigma'$.
In the following we will assume that
\begin{equation}\label{eq:cond}  h^{-1}(u) = v \cup R(h) \end{equation}
for $R(h)$ the set of ramification points of $h$. In particular, $u$ contains
all the branch points of $h$.

\begin{definition}[\textbf{Covering constructions}]\label{def:covering}

  Pulling back complex structures under the \emph{unbranched covering map}
  \[ h|_{\Sigma'-h^{-1}(u)}: \Sigma'-h^{-1}(u) \to \Sigma - u \]
   induces a holomorphic isometric embedding
  \[ f_h: \Tt(\Sigma,u) \to \Tt(\Sigma', h^{-1}(u)) \]
  with respect to the Teichm\"uller metric. We call $f_h$ a
  \emph{totally marked covering construction}. \\
  A \emph{covering construction} is a holomorphic map
  \[ f: \Tt(\Sigma,u) \to \Tt(\Sigma', v) \]
  induced by a totally marked covering construction $f_h$ by postcomposition
  with the forgetful map $\mathcal F:\Tt(\Sigma',h^{-1}(u)) \to \Tt(\Sigma',v)$.
  Since $\mathcal F$ only forgets ramification points of $h$, the map $f$ is an isometric
  embedding~\cite{BS}*{Proposition A.2}. If the branched cover $h$
  is regular we call $f$ a \emph{regular covering construction}.
\end{definition}

\begin{remark}
  In \cite{BS}*{Appendix} it is shown that for $g(\Sigma)\geq 1$ and $\Tt(\Sigma,u) \not= \Tt_{1,1}$
the condition given by \Cref{eq:cond} is necessary and sufficient for $h$ to induce
an isometric embedding with respect to the Teichm\"uller metric.
\end{remark}

In the following, we describe in more detail properties of each
type of covering construction. Unless otherwise specified
$h:(\Sigma',v) \to (\Sigma,u)$ will denote a finite degree, orientation-preserving,
topological branched cover. Let $g(\Sigma)$ denote the genus of $\Sigma$. Throughout this section we assume that
\[ 3g(\Sigma) - 3 + |u| > 0 .\]

\subsection{Totally marked covering constructions}
Let $v = h^{-1}(u) \subset \Sigma'$ and let
\[ f:\Tt(\Sigma,u) \to \Tt(\Sigma',v)\] be the associated
totally marked covering construction.
The map $f$ is holomorphic
and the derivative is given by pulling back Beltrami forms to the cover.
We sketch the details. Let $X= \HH^2/\Gamma \in \Tt(\Sigma,u)$, for $X$
a finite type Riemann surface diffeomorphic to $\Sigma - u$ and
$\Gamma < \PSL(2,\RR)$ a discrete subgroup without elliptic elements.
Since $h$
is unbranched over $\Sigma-u$, the image $f(X)$ is given by $f(X) = \HH^2/\Gamma'$ for
$\Gamma'<\Gamma$ a finite index subgroup.
It follows that there exists a pushforward lift of Beltrami forms:
\[ f_*: M(\Gamma) \to M(\Gamma') \]
given by the inclusion.
Furthermore, $f_*$ induces a holomorphic map
\[ \Tt(\Sigma,u) \to \Tt(\Sigma',v)\]
which agrees with $f$.\footnote{Here it is important that $h$ is orientation preserving,
  so that it agrees with $f(X) \to X$.}
In particular the derivative of $f$ is induced by $f_*$.
Similarly, the pullback of integrable quadratic differentials is given by
the inclusion $Q(\Gamma) \to Q(\Gamma')$.

Recall that Harmonic forms $H(X)$ are given by $(z-\bar{z})^2\bar{q}$ for $q \in Q(\Gamma)$.
Then, the following corollary is immediate.
\begin{corollary}\label{lemma:harmonic} Let $f:\Tt(\Sigma,u) \to \Tt(\Sigma',v)$ be a totally marked covering
  construction. Let $X \in \Tt(\Sigma,u)$. Then, the derivative $df$ gives an inclusion
  \[ df_X: H(X) \hookrightarrow H(f(X)) .\]
  Furthermore, $df_X$ is induced by the pullback of quadratic differentials and volume
  forms under $h$.
\end{corollary}
\begin{remark} For \Cref{lemma:harmonic} it is crucial that the pullback
  of the complete hyperbolic metric on $X$ is the complete hyperbolic metric on $f(X)$.
  In particular, it does not apply when we postcompose with forgetful maps.
\end{remark}

The Weil--Petersson metric $g_{\WP}$ is easily computed for harmonic forms. Thus, we get the following.
\begin{corollary}\label{cor:metric}
  Let $f:\Tt(\Sigma,u) \to \Tt(\Sigma,v)$ be a totally marked covering construction
  induced by a branched cover $h$. Then,
  \[f^*g_{\WP} = \deg(h) g_{\WP} .\]
\end{corollary}
\begin{proof}
  Let $X \in \Tt(\Sigma,u)$.
  By \Cref{lemma:harmonic}, $df_X: H(X) \hookrightarrow H(f(X))$ is given by the pullback of quadratic
  differentials and volume forms under $h$.
  Let $\mu_i = \overline{q_i}(ds_X^2)^{-1}\in H(X)$, then:
  \begin{equation}
    \begin{aligned}
      \langle df_X(\mu_1),df_X(\mu_2) \rangle &= \int_{f(X)} df_X(\mu_1) \overline{df_X(\mu_2)}ds_{f(X)}^2\\ &=
    \int_{f(X)}  df_X(\mu_1) h^*(q_2)\\
    &= \deg(h) \int_X \mu_1 q_2 = \deg(h) \langle \mu_1, \mu_2 \rangle.
  \end{aligned}
\end{equation}
The third equality follows from the fact that $f^*|_X:Q(f(X)) \to Q(X)$ is given by the trace under $h$.
\end{proof}

\subsection{Regular H-Covers}
Let $Y$ be a closed Riemann surface, with
underlying topological surface $\Sigma'$.
Let $H$ be a finite group of conformal automorphisms of $Y$. Let $X = Y/H$ so that
$H$ defines a regular branched covering
\[ h_X:Y \to X .\]
Let $\Sigma$ be the underlying topological surface of $X$.
Let $h:\Sigma'\to \Sigma$ be the topological branched cover induced by $h_X$. As in \Cref{eq:cond}, consider
finite subsets $u \subset \Sigma$ and $v \subset \Sigma'$
such that
\[ h^{-1}(u) = v \cup R(h).\]
Then, $h$ induces a regular covering construction
\[f: \Tt(\Sigma,u) \to \Tt(\Sigma',v). \]
Suppose that $v$ is $H$-invariant. Then $H$ is a group of conformal automorphisms of
$Y':=Y - v$.
Recall that $H$ acts on $\Tt(\Sigma',v)$ by biholomorphisms.
We claim the following,
\begin{lemma}
  \[ f(\Tt(\Sigma,u)) = \Fix(H)  \]
  for $\Fix(H)$ the fixed set of $H$ in $\Tt(\Sigma',v)$.
\end{lemma}
This is already known (e.g. ~\cite{Kravetz}), but we provide a proof using the
terminology introduced in \cite{BS}.
\begin{proof}
  Let $M :=f(\Tt(\Sigma,u)).$ Note that $H$ is a finite group
  of biholomorphic isometries of the Teichm\"uller metric. Thus, $\Fix(H) \subset \Tt(\Sigma',v)$ is a totally geodesic complex
  submanifold with respect to the Teichm\"uller metric. The inclusion $M \subseteq \Fix(H)$ is evident, as $H$ can be represented
  by biholomorphisms of any $f(Z)$ for $Z \in \Tt(\Sigma,u)$.
  To prove the converse, we will consider the totally geodesic bundles $QM$
  and $Q\Fix(H)$~\cite{BS}*{Section 3}.
  Let $f(X) \in M \subseteq \Fix(H)$ and identify $H$ with a
  subgroup of biholomorphisms of $f(X)$. Recall
  that $Q_{f(X)}M \subset Q(f(X))$ equals the $H$-invariant differentials
  on $Q(f(X))$. We claim that $Q_{f(X)}\Fix(H)$ admits the same description.
  In particular, $Q_{f(X)} \Fix(H) = Q_{f(X)}M$ and we are done.

  To prove the claim, let $t \in H$ and denote by $\mathfrak t:\Tt(\Sigma',v) \to \Tt(\Sigma',v)$
  the biholomorphism induced by changing the marking by $t$. Since $\mathfrak t(f(X)) = f(X)$, the umkehr map~$\mathfrak t_!$ \cite{BS}*{Section 4} associated to
  $\mathfrak t$ acts on $Q(f(X))$ by:
  \[ \mathfrak t_!:Q(f(X)) \to Q(f(X))\ \ , \ \ q \to t^*(q) \in Q(f(X)).\]
  Any geodesic in $\Fix(H)$ is fixed by $\mathfrak t$. Thus, $q = t^*(q)$
  for any $q \in Q_{f(X)}\Fix(H)$. The claim
  follows.
\end{proof}
$H$ acts on $\Tt(\Sigma',v)$ by Weil--Petersson isometries, thus the following is immediate.
\begin{corollary}\label{cor:regular}
  Let $f:\Tt(\Sigma,u) \to \Tt(\Sigma',v)$ be a regular covering construction induced by
  a regular covering with deck group $H$. Assume that $v$ is $H$-invariant.
  Then $f(\Tt(\Sigma,u))$ is totally geodesic with respect to the Weil--Petersson metric, i.e.
  it is Weil--Petersson.
\end{corollary}

In general, we have a weaker
statement.
\begin{corollary}\label{cor:regular_almost}
  Let $f:\Tt(\Sigma,u) \to \Tt(\Sigma',v)$ be a regular covering construction, induced
  by a regular cover $h:\Sigma' \to \Sigma$.
  If $(g(\Sigma),g(\Sigma')) = (0,1)$ assume further
  that $|u| \geq 5$. Then $f(\Tt(\Sigma,u))$ is almost Weil--Petersson.
\end{corollary}
\begin{proof}
  Let $v' = v \setminus R(h) = h^{-1}(u \setminus B(h))$,
  for $B(h)$ the branch locus of $h$.
  Then, $v'$ is $H$-invariant. We claim that $3g(\Sigma') -3 + |v'| > 0$.
  This is immediate if $g(\Sigma') > 1$ so assume otherwise.

  We proceed by cases. If $g(\Sigma) =1$ the map $h$ is actually
  unramified and $|v'| = |u| \geq 1$.
  For the remaining cases we use the following facts from
  ~\cite{Miranda}*{Lemma 3.8 and discussion thereafter}:
  \begin{enumerate}
    \item The map $h$ can have at most $4$ branch points
      and less than $4$ if $g(\Sigma') = 0$.
    \item If $g(\Sigma') = 0$ and $h$ has $3$ branch points,
      then $\deg(h) \geq 4$.
  \end{enumerate}
  The claim follows. The forgetful map
\[ \mathcal F: \Tt(\Sigma',v) \to \Tt(\Sigma',v') \]
satisfies that $M'=\mathcal F(f(\Tt(\Sigma,u))$ is Weil--Petersson (\Cref{cor:regular}).
Moreover, as $\mathcal F$ only forgets ramification points,
$\mathcal F$ restricts to a biholomorphism between $f(\Tt(\Sigma,u))$ and $M'$. The result
follows.
\end{proof}

Combining \Cref{cor:regular,cor:metric}, we complete the proof of \Cref{lemma:almost_geodesic} by showing
the following.
\begin{proposition}\label{prop:totally_marked}
  Let $f:\Tt(\Sigma,u) \to \Tt(\Sigma,v)$ be a totally marked covering construction
  induced by a branched cover $h$. Then, $f(\Tt(\Sigma,u))$ is a Weil--Petersson
  complex submanifold.
\end{proposition}

\begin{proof}
  The case of totally marked regular coverings follows from \Cref{cor:regular}.
  So assume that $h$ is not regular. Let
  \[ h_2:(\Sigma'',w) \to (\Sigma',v) \]
  be the \emph{normal closure} of $h$, i.e. the regular cover induced by
  taking a finite index subgroup $\Gamma < h_*(\pi_1(\Sigma'-v)) < \pi_1(\Sigma - u))$ such that $\Gamma \trianglelefteq \pi_1(\Sigma -u)$.
  This means that there is a commutative diagram:
\begin{center}
  \begin{tikzcd}[sep =small]
    (\Sigma'',w) \ar[dd,"h_1",swap] \ar[rd,"h_2"]&\\
    & (\Sigma',v) \ar[ld, "h"]\\
    (\Sigma,u)&
  \end{tikzcd}
\end{center}
where $h_i$ for $i=1,2$ are regular covers and the marked points satisfy
  \[ h_1^{-1}(u) = h_2^{-1}(v) = w .\]
  By \Cref{cor:regular}, the induced maps
  \begin{center}
  \begin{tikzcd}[sep = small]
    \Tt(\Sigma'',w) &\\
    & \Tt(\Sigma',v) \ar[ul, "f_2",swap]\\
    \Tt(\Sigma,u)\ar[uu, "f_1"] \ar[ur, "f",swap]&
  \end{tikzcd}
\end{center}
satisfy that $f_1(\Tt(\Sigma,u))$ and $f_2(\Tt(\Sigma',v))$ are Weil--Petersson submanifolds of $\Tt(\Sigma'', w)$.
Since $f_1 = f_2 \circ f$, it follows that $f_1(\Tt(\Sigma,u)) \subset f_2(\Tt(\Sigma',v))$ and
\[ f(\Tt(\Sigma,u)) = f_2^{-1}(f_1(\Tt(\Sigma,u))) .\]
In particular, $f_1(\Tt(\Sigma,u))$ is a totally geodesic submanifold of $f_2(\Tt(\Sigma',v))$ with
respect to $g_{\WP}$.
  The map $f_2:\Tt(\Sigma',v) \to \Tt(\Sigma'', w)$ is a totally marked
  covering construction. Thus, by \Cref{cor:metric}
  it is (up to a constant)
  a Riemannian isometric embedding for $g_{\WP}$. The claim follows.
\end{proof}

\bibliography{biblio}
\end{document}